\documentclass[11pt]{article}
\usepackage{float}
\usepackage{color}
\usepackage{amssymb}
\usepackage{amsmath}
\usepackage{amsthm}
\usepackage{graphicx}

\def\'#1{\ifx#1i{\accent"13 \i}\else{\accent"13 #1}\fi}

\usepackage{url}

\def\P{{\rm I\!P}}
\def\L{{\rm I\!L}}
\def\l{{\rm I\!l}}

\newtheorem{theorem}{Theorem}
\newtheorem{lemma}{Lemma}

\newtheorem{proposition}{Proposition}
\theoremstyle{definition}
\newtheorem{definition}{Definition}

\title{On the pseudoachromatic index of the complete graph III
\thanks{Research supported by Research by CONACyT-M{\' e}xico under projects 178395, 166306, and PAPIIT-M{\' e}xico under project IN104915.}}
\author{\sc M. Gabriela Araujo-Pardo \\ \url{garaujo@math.unam.mx} \and \sc Juan Jos{\' e} Montellano-Ballesteros \\ \url{juancho@math.unam.mx} \and \sc Christian Rubio-Montiel\\ \url{christian@math.unam.mx} \and \sc Ricardo Strausz \\ \url{dino@math.unam.mx} \and --------------------------------------------------------- \\ Instituto de Matem{\' a}ticas \\ Universidad Nacional Aut{\' o}noma de M{\' e}xico \\ Ciudad Universitaria, 04510, M{\' e}xico D.F.\\ ---------------------------------------------------------
}
\begin{document}
\maketitle

\begin{abstract} Let $ \Pi_q $ be the projective plane of order $ q $, let $\psi(m):=\psi(L(K_m))$ the pseudoachromatic number of the complete line graph of order $ m $, let $ a\in \{ 3,4,\dots,\tfrac{q}{2}+1 \} $ and  $ m_a=(q+1)^2-a $. 

In this paper, we improve the upper bound of $ \psi(m) $ given by Araujo-Pardo et al. [J Graph Theory 66 (2011), 89--97] and Jamison [Discrete Math. 74 (1989), 99--115] in the following values: if $ x\geq 2 $ is an integer and $m\in \{4x^2-x,\dots,4x^2+3x-3\}$ then
$\psi(m) \leq 2x(m-x-1)$.

On the other hand, if $ q $ is even and there exists $ \Pi_q $ we give a complete edge-colouring of $ K_{m_a} $ with $(m_a-a)q$ colours.  Moreover, using this colouring we extend the previous results for $a=\{-1,0,1,2\}$  given by Araujo-Pardo et al. in [J Graph Theory 66 (2011), 89--97] and  [Bol. Soc. Mat. Mex. (2014) 20:17--28]
proving that $\psi(m_a)=(m_a-a)q$ for $ a\in \{3,4,\dots,\left\lceil \frac{1+\sqrt{4q+9}}{2}\right\rceil -1 \} $.

\end{abstract}

\section{Introduction}

The \emph{pseudoachromatic number} of a graph $\psi(G)$, which is the number of colours in a maximum complete vertex-colouring of $G$, has attracted the attention of several researchers since its introduction by Gupta \cite{G} in 1969 (see also \cite{MR0439672,MR0357198}). Being a hard parametre to calculate, it is in order to search for bounds in general classes of graphs. In this series of papers \cite{AMS,AMRS}, we endeavour to calculate the exact value of the \emph{pseudoachromatic index} of the complete graph $K_m$, for a wide set of values of $m$ (see also \cite{MR0384589,B,HPW,J,S,TRJL}), which we denote by 
    \[\psi(m):=\psi(L(K_m))\]

More precisely, on the one hand we study in detail the interaction of a couple of natural upper bounds for $\psi(m)$ in terms of the size of chromatic classes in a complete edge colouring of $K_m$; namely, if $x$ is the size of its smallest chromatic class, then (cf., \cite{J})
    \[\psi(m)\leq{\rm min}\left\{2x(m-x-1)+1, \left\lfloor \frac{\binom{m}{2}}{x+1}\right\rfloor\right\}.\]
Each of these bounds by themselves does not give much information of the phenomena, but together they seem to dominate precisely what is going on. Therefore, we study those intervals of values of $m$ where each of these bounds dominates the phenomena; viz., denoting by
    \[g_m(x):=2x(m-x-1)+1\qquad {\rm and }\qquad f_m(x):=\left\lfloor \frac{\binom{m}{2}}{x+1} \right\rfloor,\]
it will be proven the following

\begin{theorem}
\label{cor3} 
If $x\geq2$ is an integer, then
\[
\psi(m)\leq\begin{cases}
g_{m}(x)-1 & m\in\left\{ 4x^{2}-x,\dots,4x^{2}+3x-3\right\} \\
g_{m}(x) & m\in\left\{ 4x^{2}+3x-2,4x^{2}+3x-1\right\} \\
f_{m}(x) & m\in\left\{ 4x^{2}+3x,\dots,4(x+1)^{2}-(x+1)-1\right\}.
\end{cases}
\]
\end{theorem}

On the other hand, using results given in \cite{AMS,AMRS} and supported by the combinatorial structure of projective planes of even order, we exhibit optimal complete edge colourings of $K_m$; we show that

\begin{theorem}
\label{teo1}
Let $q\geq4$ be an even natural number, and let $n=q^2+q+1$, $ a\in \{-1,0,\dots,\frac{q}{2}+1 \} $ and $ m_a=n+q-a $. If the projective plane $ \Pi_q $ of order $ q $ exists then \[\psi(m_a) \geq (m_a-a)q.\]
\end{theorem}  

With these results together, we obtain the following set of exact values for the pseudoachromatic index of the complete graph

\begin{theorem}
\label{cor2}
Let $q>4$ be a power of $ 2 $, let $n=q^2+q+1$, $ a\in \{-1,0,\dots,\left\lceil \frac{1+\sqrt{4q+9}}{2}\right\rceil -1 \} $ and $ m_a=n+q-a $, then \[\psi(m_a) = (m_a-a)q .\]
\end{theorem}

\section{Proof of Theorem \ref{cor3}}

In order to give the proof we will use the following proposition (see  \cite{J}).
\begin{proposition}
\label{prop1} 
If $x\geq2$ is an integer, then
\[
\psi(m)\leq\begin{cases}
g_{m}(x) & m\in\left\{ 4x^{2}-x,\dots,4x^{2}+3x-1\right\} \\
f_{m}(x) & m\in\left\{ 4x^{2}+3x,\dots,4(x+1)^{2}-(x+1)-1\right\}.
\end{cases}
\]
\end{proposition}

\begin{proof}[Proof of Theorem \ref{cor3}]

Let $ x\geq2 $ an integer  and let $n\in\{4x^{2}- x,\dots,4x^2 +3x-3\}$. By Proposition \ref{prop1}, we already know that $ \psi(n)\leq g_n(x) $. We will prove that $ \psi(n) \leq g_n(x)-1 $. To do this we suppose that $ \psi(n)=g_n(x) $ and finally arrive to a contradiction. Let $ \varsigma \colon V \rightarrow \{ 1,\dots,g_n(x) \} $ be a complete colouring. First of all we will prove that any class of colour can not have less than $x$ edges. Suppose there exists a colour class $C$ with $s$ edges such that $s<x$, then $C$ will be adjacent to at most $ \tbinom{s}{2}-s+2s(n-2s)=2s(n-s-1) $ edges, but $ 2s(n-s-1)<2x(n-x-1) $, in consequence $C$ could not be adjacent to all other colour classes. Then, each colour class has at least $x$ edges.  Suppose now that there exists a colour class $C$ with exactly $x$ edges. Then it is clear that $C$ is adjacent to exactly $ 2x(n-x-1) $ other edges and also they must all have different colours,  otherwise $C$ does not meet all the other colour classes --note that the only way to get this is when C is a matching. Since each colour class has at least $x$ edges, then the number of colour classes with more than $x$ edges is at most $ {\tbinom{n}{2}} - xg_n(x) $, hence, there are at least $ g_n(x)-\{ {\tbinom{n}{2}} - xg_n(x) \} $ colour classes with $ x $ edges.  Now we will see that there are at least two colour classes of size $x$. For this just observe that

$ 2 \leq g_n(x)-\{ {\tbinom{n}{2}} - x(g_n(x))\} $ if and only if $ n^2-(4x^2+4x+1)n+4x^3+8x^2+2x+2 \leq 0 $, \[ \textrm{i.e., }\left(n-\frac{4x^2+4x+1-\sqrt{D_1}}{2}\right)\left(n-\frac{4x^2+4x+1+\sqrt{D_1}}{2}\right) \leq 0 \]
where $ 4x^2+2x-3/2<\sqrt{D_1}=\sqrt{16x^4+16x^3-8x^2-7}<4x^2+2x-1$, which is equivalent to $\sqrt{D_1}=4x^2+2x-3/2+\epsilon$ for some $0<\epsilon<1/2$ and then 
\[ n\in \left[ x+\frac{5}{4}-\frac{\epsilon}{2},4x^2+3x-\frac{1}{4}+\frac{\epsilon}{2} \right] \cap \{ 4x^2-x,\dots,4x^2+3x-3 \}=\{ 4x^2-x,\dots,4x^2+3x-3\} \]

Let $C$ be a colour class of size $x$. This class is a matching with $2x$ vertices and $ {\tbinom{2x}{2}} - x $ edges in the induced subgraph $<C>$ that are not in $C$. Therefore, each one of these edges has a different colour and each one of these colours is in a class with more than $x$ edges because they are adjacent to two edges of $C$.

\begin{figure}[!htbp]
\begin{center}
\includegraphics{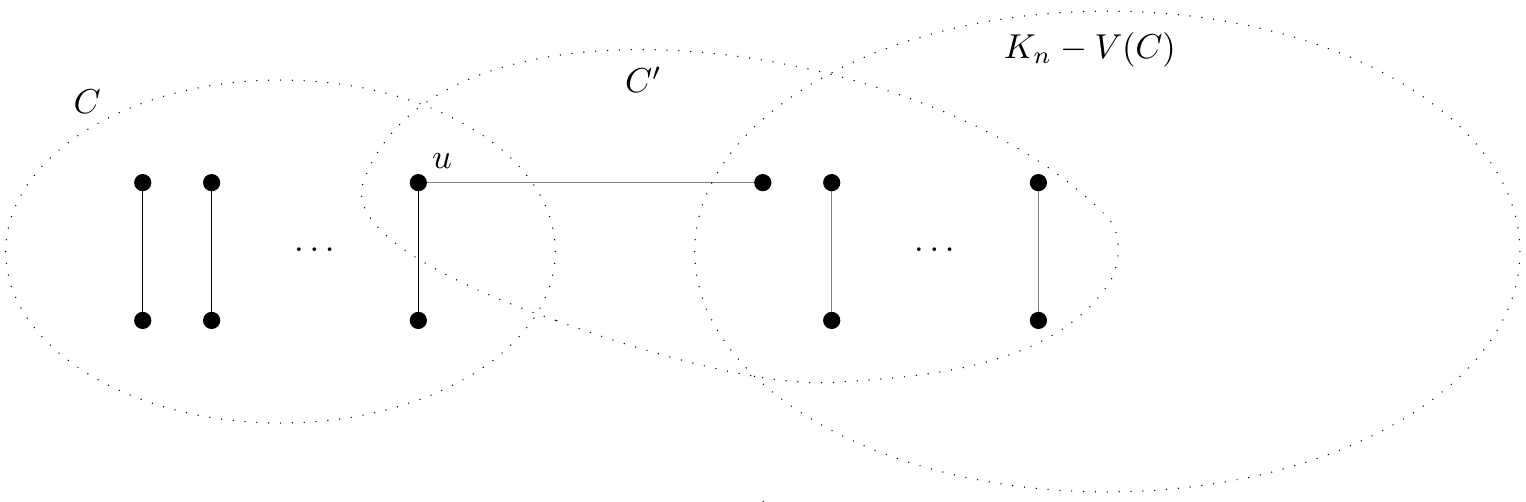}
\caption{\label{Fig1} $K_{n}$}
\end{center}
\end{figure}

Let $C'$ be another colour class of size $x$. $C'$ meets $C$ in a vertex $u$. In Fig \ref{Fig1} we give a description of $K_{n}$. The rest of the $ 2x-2 $ edges meets $ C' $ through $u$ in only one vertex and they have different colours and also their colour classes are larger than $x$ because they meet two vertices of $ C' $. As before there are at least $ g_n(x)-\{ {\tbinom{n}{2}} - x(g_n(x))\}  $ colour classes of size $x$ then there are at most $ (2x-2)(g_n(x)-\{ {\tbinom{n}{2}} - x(g_n(x))\} -1) $ chromatic classes of size greater than $x$ and, hence, we have the following:

\[ {\tbinom{n}{2}} - x(g_n(x)) \geq {\tbinom{2x}{2}} - x + (2x-2)(g_n(x)-\{ {\tbinom{n}{2}} - x(g_n(x))\} -1)\]

Therefore,
\[ (2x-1)n^2-(8x^3+4x^2-6x-1)n+8x^4+12x^3-12x^2-2x \geq 0\]
\[ \textrm{i.e., }\left(n-\frac{8x^{3}+4x^{2}-6x-1-\sqrt{D_2}}{4x-2}\right)\left(n-\frac{8x^{3}+4x^{2}-6x-1+\sqrt{D_2}}{4x-2}\right) \geq 0 \]

where $ \sqrt{D_2}=\sqrt{64x^6-144x^4+80x^3-4x^2+4x+1}=8x^3-9x+5+r_2 $ and then 

\[ \left(n-(x+\frac{5}{4}+r_3)\right)\left(n-(4x^2+3x-\frac{9}{4}+r_4)\right) \geq 0 \]

\[ \textrm{i.e., }n \in \left[4x^{2}+3x-\frac{9}{4}+r_4,\infty \right) \cap \{ 4x^2-x,\dots,4x^2+3x-3 \}= \emptyset \]

Then we have a contradiction and we conclude that if $x\geq2$ is an integer and $n\in \{4x^2-x,\dots,4x^2+3x-3\}$ then
\[\psi(n) \leq g_n(x)-1\]
\end{proof}

\subsection{Proof of Theorem \ref{teo1}.}

In order to prove Theorem \ref{teo1} we only need to show that $\psi(m_a)\geq (m_a-a)q$. We will do this by exhibiting a complete edge-colouring  of $K_{m_a}$ with $(m_a-a)q$ colours.

For the construction of such an edge-colouring,  we need some definitions and remarks.

A \emph{projective plane} consists of a set of $n$ points, a set of lines, and an incidence relation between points and lines having the following properties:
\begin{enumerate} \item Given any two distinct points there is exactly one line incident with both of them.
\item Given any two distinct lines there is exactly one point incident with both of them.
\item There are four points, such that no line is incident with more than two of them. \end{enumerate}
Such plane has $n= q^2 + q+ 1$ points (for some number $q$) and $n$ lines; each line contains $q+1$ points and each point belongs to $q+1$ lines. The number $q$ is called the \emph{order} of the projective plane. A projective plane of order $q$ is called $\Pi_q$. If $q$ is a prime power there exists $\Pi_q$, which is called the \emph{algebraic projective plane} since it arises from finite fields.

Let $\P$ be the set of points of $\Pi_q$ and let $\L= \{\l_1,\dots, \l_n\}$ be the set of lines of $\Pi_q$. Now, we will identify the points of $\Pi_q$ with the set of vertices of the complete graph $K_n$.  In a natural way, the set of points of each line of $\Pi_q$  induces a subgraph isomorphic to $K_{q+1}$ in $K_n$.  For each line $\l_i\in \L$, let $l_i = (V(l_i), E(l_i))$ be the subgraph of $K_n$ induced by the set of $q+1$ points of $\l_i$.  By the properties of the projective plane, for every pair $\{i,j\}\subseteq \{1,\dots,n\}$, $|V(l_i)\cap V(l_j)| =1$ and $\{E(l_1), \dots, E(l_n)\}$ is a partition of $E(K_n)$.  In this way, when we say that a graph $G$ isomorphic to $K_n$ is a \emph{representation of the projective plane} $\Pi_q$, we will understand that $V(G)$ is identified with the  points of $\Pi_q$ and that there is  a family of subgraphs (lines) $\{l_1, \dots, l_n\}$ of $G$ such that for each  line $\l_i$  of $\Pi_q$, $l_i$ is the subgraph induced by the set of points  of $\l_i$.

Given two graphs $G$ and  $H$,  the directed sum, $G\oplus H$, is defined as the graph with vertex set $V(G)\cup V(H)$ and edge set $E(G)\cup E(H)\cup (V(G) \times V(H))$ (the set of edges $V(G) \times V(H)$ for short; we write it as $V(G)V(H)$-edges). Given $S\subseteq V(G)$, $G\setminus S$ is the subgraph of $G$ induced by $V(G)\setminus S$.  

Let $m$ be a positive integer. Given an edge-colouring $\Gamma\colon E(K_m)\rightarrow \mathcal{C}$, we will say that a vertex $x\in V(K_m)$ is an \emph{owner} of a set of colours $\mathcal{C}'\subseteq \mathcal{C}$  whenever for every $c\in \mathcal{C}'$ there is $y\in V(K_m)$ such that $\Gamma(xy) = c$; and given a subgraph $G$ of $K_m$, we will say that $G$ is an \emph{owner} of a set of colours   $\mathcal{C}'\subseteq \mathcal{C}$ if each vertex of $G$ is an owner of $\mathcal{C}'$. By this, $\Gamma$ is a complete edge-colouring if for every pair of colours in $\mathcal{C}$ there is a vertex in $K_m$ which is an owner of both colours.  

\begin{lemma}\label{owner} Let  $n= q^2 + q+ 1$, with $q$ a natural number such that $ \Pi_q $ exists, and  let $t$ be a positive integer. Let $G$ be a subgraph  of $K_{n+t}$  isomorphic to $K_n$ and let $G$ be a representation of $\Pi_q$.  Let $\Gamma\colon E(K_{n+t})\rightarrow \mathcal{C}$ be an edge-colouring of $K_{n+t}$. Suppose that each line $l_i$  of $G$ is an owner of a set of colours $\mathcal{C}_i\subseteq \mathcal{C}$. Then for every pair of colours $\{c_1, c_2\}\subseteq \bigcup\limits_{i=1}^{n} \mathcal{C}_i$ there is $x\in V(G)$ which is an owner of $c_1$ and $c_2$.\end{lemma}
\begin{proof} Let $\{c_1, c_2\}\subseteq \bigcup\limits_{i=1}^{n} \mathcal{C}_i$. If there is $i\in \{1,\dots,n\}$ such that $\{c_1, c_2\}\subseteq \mathcal{C}_i$, then since $l_i$ is an owner of $\mathcal{C}_i$ it follows that each $x\in V(l_i)$ is an owner of $c_1$ and $c_2$. If $c_1\in \mathcal{C}_i$ and $c_2\in\ \mathcal{C}_j$ ($i\not=j$), there is $x\in V(G)$ such that $x= V(l_i)\cap V(l_j)$, and then since $l_i$ and $l_j$ are owners of $\mathcal{C}_i$ and $\mathcal{C}_j$, respectively, $x$ is an owner of $c_1$ and $c_2$.\end{proof}
 
Now, we define different edge-colourations for some special graphs that will be used later. 

It is well known (see \cite{BM}, \cite{H}) that any complete graph of even order $ r $ admits a $ 1 $-factorization and that any complete graph of odd order $ r $ admits a $ 2 $-factorization by Hamiltonian cycles.

\begin{definition} Let $r$ be an even integer. An edge-colouring $\Gamma\colon E(K_r) \rightarrow \{1,2,\dots,r-1\}$ will be said to be of \emph{Type 1} if for every $i\in \{1,2,\dots,r-1\}$ the set $\{ xy\in E(K_r) \colon \Gamma(xy) = i\}$ is a perfect matching of $K_r$.\end{definition}

\begin{definition} Let $r$ be an odd integer. An edge-colouring $\Gamma \colon E(K_r) \rightarrow \{1,\dots,r\}$ will be said to be of \emph{Type 2}  if we obtain $\Gamma$ in the following way: Let $G$ be the graph (isomorphic to $K_{r+1}$) obtained by adding to $K_r$ a new vertex $x_0$  and all the $x_0V(K_r)$-edges. Let $\Gamma'$ be an edge-colouring of Type 1 of $G$ and, for every $e\in E(K_r)$, let $\Gamma(e) := \Gamma'(e)$. \end{definition}

\begin{definition} Let $r$ be an odd integer and $x,y\in V(K_r)$. An edge-colouring $\Gamma \colon E(K_r-xy) \rightarrow \{1,\dots,r-2\}$ will be said to be  of \emph{Type 3} if we obtain $\Gamma$ in the following way: Let $G$ be the graph (isomorphic to $K_{r-1}$) obtained by deleting the vertex $x$  and all the $xV(K_{r-1})$-edges. Let $\Gamma'$ be an edge-colouring of Type 1 of $G$ and, for every $e\in E(K_r-xy)$, let $\Gamma(e) := \Gamma'(e)$ if $ e\in E(G) $, and $ \Gamma(xw) := \Gamma'(yw)$ for every  $ w\in V(G)-y $. \end{definition}

\begin{definition} Let $r$ be an odd integer. An edge-colouring $\Gamma_i \colon E(C_r) \rightarrow \{i,i+\frac{r-1}{2}\}$ will be said to be of \emph{Type 4} if we obtain $\Gamma_i$ in the following way: Let $G$ be the graph (isomorphic to $P_{r}$) obtained by deleting the edge $x_0y\in E(C_r)$. Let $\Gamma'_i\colon E(G) \rightarrow \{i,i+\frac{r-1}{2}\}$ be a proper edge-colouring of $G$ (remember that \emph{proper} means that each vertex has different colours in its edges) and, for every $e\in E(C_r)$, let $\Gamma_i(e) := \Gamma'_i(e)$ be if $ e\in E(G) $, and $ \Gamma_i(x_0y) := \Gamma'_i(x_0w)$ for $ w = N(x_0)-y $. Observe that $ x_0 $ is an owner of one colour. \end{definition}

\begin{definition} Let $r$ be an odd integer. An edge-colouring $\Gamma \colon E(K_r) \rightarrow \{1,\dots,r-1\}$ will be said to be of \emph{Type 5} in $ x_0 $ if we obtain $\Gamma$ in the following way: Let $ \{G_1,\dots,G_{\frac{r-1}{2}}\} $ be a $ 2 $-factorization of $ K_r $ such that $ G_i=C_r $ for each $ i\in \{ 1,\dots,\frac{r-1}{2} \} $ and $ x_0 $ is the same in each $ G_i $. Let $\Gamma_i$ be a edge-colouring of $G_i$  of Type 4 and, for every $e\in E(K_r)$, let $\Gamma(e) := \Gamma_i(e)$ be if $ e\in G_i $. Observe that $ x_0 $ is an owner of $ \tfrac{r-1}{2} $ colours. \end{definition}

\subsection{The edge-colouring}

\begin{proof}[Proof of Theorem \ref{teo1}] \label{colo1}
To prove this theorem we will exhibit a complete edge-colouring of $K_{m_a}$ with $(m_a-a)q$ colours. The cases for $a\in \{-1,0,1,2\}$ are given in \cite{AMS,AMRS}.  Let $ a\in \{3,4,\dots,\tfrac{q}{2}+1 \} $. Let $\mathcal{C}$ be a set of $(m_a-a)q$ colours and let $\{\mathcal{C}_1, \mathcal{C}_2,\dots, \mathcal{C}_{n}\}$ be a partition of $\mathcal{C}$ in the following way: $\mathcal{C}_i$ is a set of $q$ colours, for $1\leq i \leq q-2a+3$; $ \mathcal{C}_i $ is a set of $ q-1 $ colours, for $ q-2a+4\leq i \leq a(q-1)+q+1 $; $\mathcal{C}_i$ is a set of $q+1$ colours, for $a(q-1)+q+2\leq i \leq q^2 + q$ and $ \mathcal{C}_{n} $ is a set of $ q-1 $ colours. 

Let $G$ be a subgraph of $K_{m_a}$ isomorphic to $K_n$ and let $H = K_{m_a}\setminus V(G)$. Clearly $H$ is isomorphic to $K_{q-a}$ and $K_{m_a} = G\oplus H$.  Let $G$ be a representation of $\Pi_q$ and let $L=\{l_1, \dots, l_{n}\}$ be the set of lines of $G$.

Let $V(H)=\{h_1, \dots, h_{q-a}\}$, let $v_0\in V(G)$ and let $\ell$ be a line $l$ of $G$ such that $v_0\notin V(\ell)$. 

Let $ W $, $ U $ and $ V $ be a partition of $ V(\ell) $ such that $ W=\{w_1,\dots,w_{q-2a+3}\}$, $ U=\{u_1,\dots,u_{a-2}\} $ and $ V=\{v_1,\dots,v_a\} $, then $\ell=<W>\oplus <U>\oplus <V>$. Let $ L_0= \{ \ell_{x} \colon x\in \ell \hbox{ and } v_0\in\ell_x  \} \subseteq L $. Let $ L_W $, $ L_U $ and $ L_V $ be a partition of $ L_0 $ such that $ L_W = \{ \ell_x \colon x \in W \} $, $ L_U = \{ \ell_x \colon x \in U \} $ and $ L_V = \{ \ell_x \colon x \in V \} $.

For $ i\in \{ 1,\dots,a \} $, let $ L_i = \{ \ell^{v_i}_{j}\colon v_i\in\ell^{v_i}_{j},j\in \{ 1,\dots,q-1 \} \}$ be the set of lines $ l\not=\ell $ such that $ l\notin L_0 $. For $ i\in \{ 1,\dots,a \} $ and $ j\in \{ 1,\dots,q-1 \} $, let $ Z= \{z^{v_i}_j = \ell^{v_i}_j \cap \ell_{v_{i+1}} \colon i\not=a \} \cup \{ z^{v_a}_j = \ell^{v_a}_j \cap \ell_{v_{1}} \} \subseteq V(G) $ and let $ Y = V(G)- Z\cup V(\ell) \cup \{v_0\} $.

Without loss of generality, let $ l_i = \ell_{w_i}$ for $ i\in \{ 1,\dots,q-2a+3 \} $, $ l_{i+q-2a+3}=\ell_{u_i} $ for $ i\in \{ 1,\dots,a-2 \} $, $ l_{i+q-a+1}=\ell_{v_i} $ for $ i\in \{ 1,\dots,a \} $ and let $ l_{i(q-1)+2+j} = \ell^{v_i}_j $ for $ i\in \{ 1,\dots,a \} $ and $ j\in \{ 1,\dots,q-1 \} $, $ L' = \{ l_{i}\colon a(q-1)+q+2\leq i \leq q^2+q \}$ and $ \ell=l_{n} $. In Fig \ref{Fig2} we give a description of $K_{m_a}$.

\begin{figure}[!htbp]
\begin{center}
\includegraphics{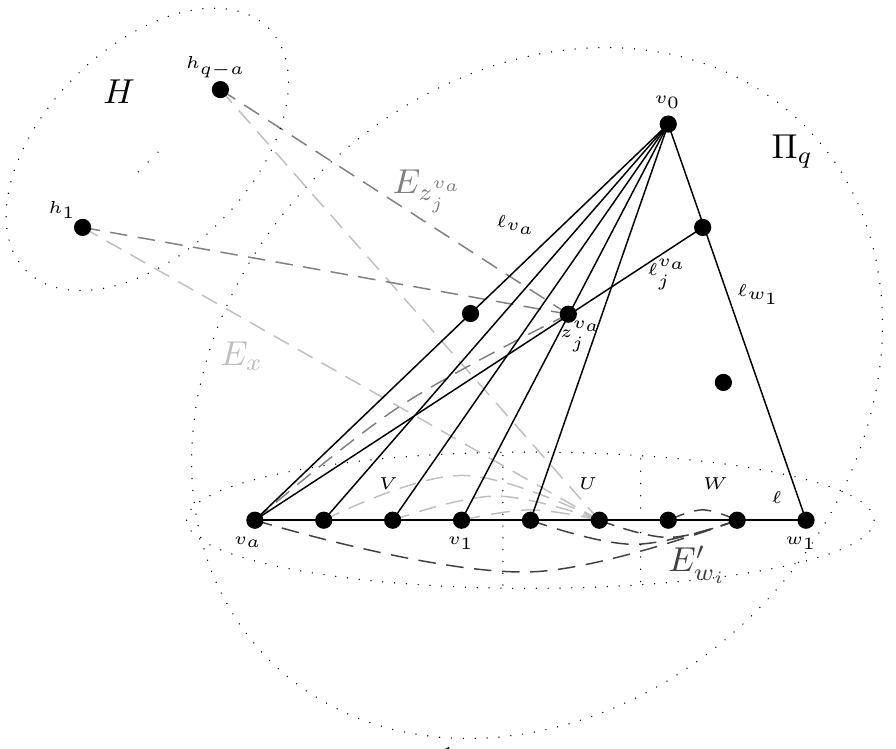}
\caption{\label{Fig2} $K_{m_a}$.}
\end{center}
\end{figure}

With the aim of defining some subset of edges of $ K_{m_a} $ we first define a special function $ h $:

For $ i\in \{1,\dots,q-2a+3\} $ and $ j\in \{ 1,\dots,\tfrac{q}{2}-a+1 \} $, let \[ h(i+j)=\begin{cases}
i+j & \textrm{if }i+j\leq q-2a+3\\
i+j-(q-2a+1) & \textrm{if }i+j>q-2a+3
\end{cases} \]
 and let \[ E'_{w_i}=\{ w_iw_{h(i+1)},\dots,w_iw_{h(i+ \frac{q}{2}-a+1)},w_iu_1,\dots,w_iu_{a-2},w_iv_a\} \] be a set of $ \tfrac{q}{2} $ edges. 

Now, we will define other subsets of edges of $ K_{m_a} $. 

For $ x \in U\cup W $, let \[ E_{x}=\{ xv_1,\dots,xv_{a-1},xh_1,\dots,xh_{q-a} \} \] be a set of $ q-1 $ edges. For each $ z^{v_i}_j $, let \[ E_{z^{v_i}_j}=\{ z^{v_i}_jv_i, z^{v_i}_jh_1,\dots,z^{v_i}_jh_{q-a} \} \] be a set of $ q-a+1 $ edges. Let \[ E'=\{ v_0v_a,v_0u_1,\dots,v_0u_{a-2},v_au_1,\dots,v_au_{a-2} \}\cup E(<U>) \] be a set of $ \tbinom{a}{2} $ edges.

We begin by colouring the edges of $K_{m_a}$ in the following way:

\begin{enumerate}

\item For $l_i\in L_W$, let $\Gamma_i\colon E(l_i)\rightarrow \mathcal{C}_i$ be an  edge-colouring of Type 5 in $ w_i $ and let $\mathcal{C}(w_i)$ be the subset of $ \frac{q}{2} $ colours of $ \mathcal{C}_i$ which $ w_i $ is not an owner then we assign exactly the colours of $ \mathcal{C}(w_i) $ to the set $E'_{w_i}$.

In this way, each line $l_i\in  L_W$ is an owner of $\mathcal{C}_i$, and we have assigned a colour to each edge of $\bigcup\limits_{i=1}^{q-2a+3}(E(l_i)\cup E'_i)$.

\item For each $l_i\in L_U\cup L_V$ let $l_i'$ be the subgraph of $G$ obtained by deleting the edge $v_0u_i$ from $l_i$ if $ i\in \{ q-2a+4,\dots,q-a+1 \} $ and the edge $v_0v_i$ from $l_i$ if $ i\in \{ q-a+2,\dots,q+1 \} $. Let $\Gamma_i\colon E(l_i')\rightarrow \mathcal{C}_i$ be an  edge-colouring of Type 3. 

For each $l_j\in L_i$, let $l_j'=l_j\setminus E_{z^{v_i}_j}$ be and let $\Gamma_j\colon E(l_j')\rightarrow \mathcal{C}_j$ be an  edge-colouring of Type 3. 

Now, each line $l_j$ in  $L_{i}$ is an owner of $\mathcal{C}_j$, and at this point we have assigned a colour to each edge of $\bigcup\limits_{i=1}^{q-2a+3}(E(l_i)\cup E'_i) \cup \bigcup\limits_{i=q-2a+4}^{a(q-1)+q+1}(E(l'_i)$.

\item For each $l_i\in L'$ let $\Gamma_i\colon E(l_i)\rightarrow \mathcal{C}_i$ be an  edge-colouring of Type 2. For each $l_i\in L'$, and for each $x\in V(l_i)$, let $c(x, l_i)$ be the only colour $c\in \mathcal{C}_i$ such that for every $y\in V(l_i-x)$, $\Gamma_{i}(xy)\not=c$. Observe that $\bigcup\limits_{x\in V(l_i)} c(x, l_i) = \mathcal{C}_i$. For each $x$ in $ L' $ let $c(x) = \{ c(x, l_i) \colon x\in V(l_i) \hbox{ and } l_i\in L'\}$. 

For each $ y $ in $ Y $ there are $a+1$ lines $l\notin L'$ such that $y\in V(l)$, then $c(y)$ is a set of $q-a$ colours. Colour the set of $q-a$  edges $\{yh_1,\dots,yh_{q-a} \}$ with the set of  colours $c(y)$.

For each $ z $ in $ Z $ there are $a$ lines $l\notin L'$ such that $z\in V(l)$, then $c(z)$ is a set of $q-a+1$ colours. Colour the set of $q-a+1$ edges $E_z$ with the set of  colours $c(z)$.

For each $ x \in U\cup V $ there are $2$ lines $l\notin L'$ such that $x\in V(l)$, then $c(x)$ is a set of $q-1$ colours. Colour the set of $q-1$ edges $E_x$ with the set of  colours $c(x)$.

Now it just remains to assign colours to the edges $ H\oplus<V>\oplus \{ v_0 \} $ and $ E' $.

\item Let $H' = H\oplus<V>\oplus v_0$ be and let $\Gamma_{n}\colon E(H'-v_0v_a)\rightarrow \mathcal{C}_{n}$ be an edge-colouring of Type 3. In this way, $H'$ is an owner of $ \mathcal{C}_{n}$. 

\item Let $\Gamma\colon E'\rightarrow \{c\}$ be an edge-colouring where $ c \in \mathcal{C}$.

We have already assigned a colour to each edge in $K_{m_a}$. If $\{c_1, c_2\}\subseteq \bigcup\limits_{i=1}^{n-1} C_i$, then since every line $l_i$ in $G$ is an owner of $C_i$, by Lemma \ref{owner} it follows that there is $x\in V(G)$ which is an owner of both colours. Analogously, if $\{c_1, c_2\}\subseteq C_{n}$, since $H'$ is an owner of $C_{n}$, there is $x\in V(H')$ which is an owner of both colours. Let us suppose  $c_1\in \bigcup\limits_{i=1}^{n-1} C_i$ and $c_2\in C_{n}$. If $c_1\in C_j$ with $ 1 \leq j \leq a(q-1)+q+1$, there is a vertex $ x \in V\cap V(l_j) $ and $x$ is an owner of $c_1$, and since $x\in V(H')$, $x$ is also an owner of $c_2$. If $c_1\in C_j$ with $a(q-1)+q+2\leq j\leq q^2+ q$, there is a vertex $x\in V(l_j)$ such that  $c(x, l_j)= c_1$ and, by construction, there is $y\in V(H')$ such that $\Gamma_j(xy)= c_1$. Hence $y$ is an owner of $c_1$ and since $y\in V(H')$, $y$ is an owner of $c_2$. Therefore, $\Gamma$ is a complete edge-colouring of $K_{m_a}$ and the theorem follows.
\end{enumerate}
\end{proof}

\section{Proof of  Theorem \ref{cor2}.}

In order to prove Theorem \ref{cor2} we need to show that $ \psi(m_a) \leq (m_a-a)q$ for $ a \in \{ 3,\dots, \left\lceil \frac{1+\sqrt{4q+9}}{2}\right\rceil -1\} $, newly the cases for $a\in \{-1,0,1,2\}$ are given in \cite{AMS,AMRS}.

To begin with, we prove the  following lemma.

\begin{lemma}\label{upper} Let $q$ be an even natural number, let $ a\in \{ 3,4,\dots,\frac{q}{2}+1 \} $ and $ m_a=(q+1)^2-a $, then \[\psi(m_a) \leq f_{m_a}(\frac{q}{2}).\]\end{lemma}

\begin{proof} If $ x=\frac{q}{2} $ then $ m_{\frac{q}{2}+1}=4x^2+3x $ and  $ m_{3}+\frac{3}{2}q+4=4(x+1)^2-(x+1)-1 $. By Theorem \ref{cor3} the lemma follows.\end{proof}

Finally, we prove the following.

\begin{proof}[Proof of Theorem \ref{cor2}]

By lemma \ref{upper} we know that $\psi(m_a)\leq f_{m_a}(\tfrac{q}{2})$, and \[ f_{m_a}(\tfrac{q}{2})= \left\lfloor \tfrac{q^4+4q^3+(5-2a)q^2+2(1-a)q+a^2-a}{q+2} \right\rfloor = (m_a-a)q + \left\lfloor \frac{\binom{a}{2}}{\frac{q}{2}+1} \right\rfloor\]

On the other hand, \[ \left\lfloor \frac{\binom{a}{2}}{\frac{q}{2}+1} \right\rfloor = 0 \Leftrightarrow \frac{a^2-a}{q+2} < 1  \Leftrightarrow a^2-a-(q+2)<0 \Leftrightarrow \left(a-\frac{1-\sqrt{4q+9}}{2}\right)\left(a-\frac{1+\sqrt{4q+9}}{2}\right)<0 \]

so, $ \frac{1-\sqrt{4q+9}}{2} < a < \frac{1+\sqrt{4q+9}}{2} $ and then $ a \in \{ 3,\dots, \left\lceil \frac{1+\sqrt{4q+9}}{2}\right\rceil -1\} $. 

By Theorem \ref{teo1}, it follows that $(m_a-a)q\leq \psi(m_a)\leq (m_a-a)q$ and the result follows.
\end{proof}


\end{document}